\documentclass{amsart}[12pt]
\def\doctype{}

\usepackage{latexsym,amssymb}
\usepackage{fancyhdr}
\usepackage{lineno}
\usepackage{hyperref}

\newcommand\lam{\lambda}

\newcommand{\cB}{\mathcal{B}}

\newcommand\Q{\mathbb{Q}}
\newcommand\Z{\mathbb{Z}}
\newcommand\F{\mathbb{F}}
\newcommand\obf{\bar{f}}
\newcommand\LP{\mathrm{LP}}
\newcommand\cE{\mathcal{E}}
\newcommand\cF{\mathcal{F}}
\newcommand\cL{\mathcal{L}}

\newcommand{\comment}[1]{}

\numberwithin{equation}{section}

\fancypagestyle{titlepage}{
\fancyhead[R]{\doctype}
\fancyhead[CL]{}
\cfoot{\vspace{5pt} \thepage}
}

\let\oldsection\section
\newcommand\boldsection[1]{\oldsection{\bf #1}}
\newcommand\starsection[1]{\oldsection*{\bf #1}}
\makeatletter
\renewcommand\section{\@ifstar\starsection\boldsection}
\makeatother

\newtheoremstyle{theorem}
  {12pt}		  
  {0pt}  
  {\sl}  
  {\parindent}     
  {\bf}  
  {. }    
  { }    
  {}     
\theoremstyle{theorem}
\newtheorem{thm}{Theorem}[section]  
\newtheorem{lemma}[thm]{Lemma}     
\newtheorem{cor}[thm]{Corollary}

\newtheorem{cons}[thm]{Construction}
\newtheorem{prop}[thm]{Proposition}

\newtheoremstyle{definition}
  {12pt}		  
  {0pt}  
  {}  
  {\parindent}     
  {\bf}  
  {. }    
  { }    
  {}     
\theoremstyle{definition}

\renewcommand{\proofname}{Proof}

\makeatletter
\renewenvironment{proof}[1][\proofname]{\par
  \pushQED{\qed}%
  \normalfont \partopsep=\z@skip \topsep=\z@skip
  \trivlist
  \item[\hskip\labelsep
        \scshape
    #1\@addpunct{.}]\ignorespaces
}{%
  \popQED\endtrivlist\@endpefalse
}
\makeatother

\makeatletter
\renewcommand*\@maketitle{%
  \normalfont\normalsize
  \@adminfootnotes
  \@mkboth{\@nx\shortauthors}{\@nx\shorttitle}%
  \global\topskip42\p@\relax 
  \@settitle
  \ifx\@empty\authors \else {\vskip 1em
\vtop{\centering\shortauthors\@@par}} \fi
  \ifx\@empty\@date \else {\vskip 1em \vtop{\centering\@date\@@par}}\fi 
  \ifx\@empty\@dedicatory
  \else
    \baselineskip18\p@
    \vtop{\centering{\footnotesize\itshape\@dedicatory\@@par}%
      \global\dimen@i\prevdepth}\prevdepth\dimen@i
  \fi
  \@setabstract
  \normalsize
  \if@titlepage
    \newpage
  \else
    \dimen@34\p@ \advance\dimen@-\baselineskip
    \vskip\dimen@\relax
  \fi
} 
\renewcommand*\@adminfootnotes{%
  \let\@makefnmark\relax  \let\@thefnmark\relax
  \ifx\@empty\@subjclass\else \@footnotetext{\@setsubjclass}\fi
  \ifx\@empty\@keywords\else \@footnotetext{\@setkeywords}\fi
  \ifx\@empty\thankses\else \@footnotetext{%
    \def\par{\let\par\@par}\@setthanks}%
  \fi
\thispagestyle{titlepage}
}
\makeatother

\begin{document}

\title[Generalized Laminar Families]{\large Generalized Laminar Families and \\Certain Forbidden Matrices}

\author{Peter J.~Dukes}
\address{\rm Peter J.~ Dukes:
Mathematics and Statistics,
University of Victoria, Victoria, Canada
}
\email{dukes@uvic.ca}

\thanks{Research of the author is supported by NSERC}

\date{\today}

\begin{abstract}
Recall that in a laminar family, any two sets are either disjoint or contained one in the other.  Here, a parametrized weakening of this condition is introduced.  Let us say that a set system $\cF \subseteq 2^X$ is $t$-laminar if $A,B \in \cF$ with $|A \cap B| \ge t$ implies $A \subseteq B$ or $B \subseteq A$.  We obtain very close asymptotic bounds in terms of $n$ on the maximum size of a $2$-laminar family $\cF \subseteq 2^{[n]}$.   A construction for 3-laminar families and a crude analysis for general $t$ are also given.
\end{abstract}
\maketitle
\hrule

\section{Introduction}

Let $X$ be a set.  A {\em laminar} family of subsets $\cL \subseteq 2^X$ has the property that for $A,B \in \cL$, either $A \subseteq B$, $B \subseteq A$, or $A \cap B = \emptyset$.  Consider a finite set $X$, say of cardinality $n$.  (Usually $X=[n]:=\{1,\dots,n\}$ is assumed.)  Then $|\cL| \le 2n$ by an easy (strong) induction argument.  Equality is achieved if $\cL$ consists of a full chain of size $n+1$, along with singletons.  

Said another way, $\cL$ is laminar if $A,B \in \cL$ with $|A \cap B| \ge 1$, implies $A \subseteq B$ or $B \subseteq A$.  Now, let $t$ be a positive integer and say a family $\mathcal{F} \subseteq 2^X$ is $t$-\emph{laminar} if, whenever $A,B \in \cF$ with $|A \cap B| \ge t$, we have $A \subseteq B$ or $B \subseteq A$.  It is clear that an $s$-laminar family is also $t$-laminar for $t > s$.

From the point of view of obtaining upper bounds on $t$-laminar families $\cF$, there is no loss in assuming $\binom{[n]}{\le t} \subseteq \cF$.  Also, the set $[n]$ itself can be thrown into $\cF$ for free.

The author's original motivation for this topic was not to generalize laminar families.  It originated from a certain problem of Richard Anstee on forbidden configurations in zero-one matrices,  \cite{ABS} .  

In this direction, say that a zero-one matrix $M$ {\em avoids} another zero-one matrix $Z$ if no permutation of rows and columns of $Z$ appears as a submatrix of $M$.  This is especially interesting when $M$ is has a fixed number $n$ of columns and both $M$ and $Z$ have no repeated rows.  (It should be clarified that Anstee works with the transpose: $m$ rows and no repeated columns, but the difference is not too important here.)
Under the assumption of distinct rows, it follows that sufficiently large $M$ are guaranteed to contain a copy of $Z$. The key question is how many rows $M$ can have while avoiding $Z$.

The rows of $M$ can be interpreted as characteristic vectors of distinct subsets of $[n]$.  Likewise, a set system $\cF \subseteq 2^{[n]}$ defines a $|\cF| \times n$ matrix $M(\cF)$ with rows indexed by $\cF$ and 
$$M(\cF)_{A,i} = \begin{cases} 
1 & \text{if~} i \in A \in \cF, \\
0 & \text{if~} i \not\in A.
\end{cases}
$$
Now consider the matrix 
$$Z = \left[
\begin{array}{cccc}
0 & 1 & 1 & 1 \\
1 & 0 & 1 & 1 \\
\end{array}
\right].
$$
We have chosen $Z = F_{0112}^\top$ from \cite{ABS}, where in that more general context the subscripts denote the number of rows (our columns) which are 00, 01, 10, and 11.  Incidentally, among all $2 \times 4$ matrices, the forbidden configuration problem is by far least understood for our $Z$.

Observe that $M(\cF)$ contains $Z$ as a (possibly permuted) submatrix if and only if there exist four points $w,x,y,z \in [n]$ and two sets $A,B\in \cF$ so that
\begin{itemize}
\item
$x,y,z \in A$ but $w \not\in A$; and
\item
$w,y,z \in B$ but $x \not\in B$.
\end{itemize}
It follows that forbidding $Z$ corresponds to building a $2$-laminar family.  More generally, forbidding the $2 \times (t+2)$ matrix $F_{011t}^\top$ corresponds with $t$-laminar families for each positive integer $t$.

In another direction, let us recall that $2^{[n]}$ is a graded poset.  In a laminar family, we must have that every rank one element (singleton set) belongs to a unique maximal chain in the induced ordering.  Similarly, in a $t$-laminar family, every element of rank $t$ sits below a unique maximal chain.  This is potentially a natural property for other graded posets; e.g. in the number-theoretic setting ($[n]$ with divisibility) or finite geometric setting ($\F_q^k$ with subspace containment).

Summarizing the above, we claim a collection of equivalent properties for set systems.
\begin{prop}
The following are equivalent for $\cF \subseteq 2^{[n]}$:
\vspace{-11pt}
\begin{enumerate}
\item
$\cF$ is $t$-laminar;
\item
the matrix $M(\cF)$ avoids $F_{011t}^\top$;
\item
every element of $\cF \cap \binom{[n]}{t}$ belongs to a unique maximal chain in $\cF \cap \binom{[n]}{\ge t}$.
\end{enumerate}
\end{prop}

In this note, we are especially interested in asymptotic upper bounds on such families $\cF$ for the case $t=2$ and $n \gg 0$.  We obtain fairly close estimates as our main result.

\begin{thm}
\label{main}
Let $f(n)+n+1$ be the maximum size of a $2$-laminar family on an $n$-set.  Then
\begin{eqnarray*}
\mathrm{(a)}& \displaystyle\liminf_{n \rightarrow \infty} f(n) \tbinom{n}{2}^{-1} &\ge 1.3818,~\text{and} \\
\mathrm{(b)}& \displaystyle\limsup_{n \rightarrow \infty} f(n) \tbinom{n}{2}^{-1} &\le 1.3821.
\end{eqnarray*}
\end{thm}

As the reader might expect, the lower bound follows from a constructive argument.  The upper bound is obtained from a recursive linear program.  These define Sections 2 and 3 to follow.   Interestingly, obtaining even slightly better bounds is probably very difficult; indeed, determining an exact limit as above is essentially as hard as settling certain open cases for projective planes.  We are also able to say something about the case $t=3$, but (for now) this represents a barrier. 

\section{The lower bound: designs and packings}

A $t$-{\em wise balanced design} is a pair $(V,\cB)$, where $V$ is a set of $v$ {\em points} and $\cB$ is a collection of subsets of $V$ called {\em blocks} (and having size $\ge t$) such that every $t$-subset of points belongs to exactly one block.
A $t$-{\em wise packing} is defined similarly, except that some $t$-subsets of points may belong to zero blocks.

If the block sizes are a constant $k$, we speak of a $t$-$(v,k,1)$ design (or packing).  Here, the `1' denotes the (maximum) number of blocks containing a fixed $t$-subset; in more general settings this parameter may be any nonnegative integer $\lam$.

The connection with $t$-laminar families is straight from definitions.

\begin{prop}
\label{packings}
Let $\cF \subseteq 2^{[n]}$ be a $t$-laminar family.  Then any antichain in $\cF$ forms the blocks of a $t$-wise packing on $n$ points.  In particular, the maximal sets in $\cF \setminus \{[n]\}$ form a $t$-wise packing and the sets in $\cF$ of fixed cardinality $k>t$ in $\cF$ form a $t$-$(n,k,1)$ packing.
\end{prop}

A $t$-laminar family can arise from `nested' packings. 

\begin{cons}
\label{nest}
Suppose $([n],\cB)$ is a $t$-wise packing.  Define $\cF$ by replacing each block $K \in \cB$ by any copy of a $t$-laminar family on $K$.
Then $\cF$ is $t$-laminar with ground set $[n]$.
\end{cons}

\begin{proof}
For $K \in \cB$, let $\cE_K$ denote the $t$-laminar family replacing $K$ in the construction, so that
$\cF = \cup_{K \in \cB} \cE_K$.  Consider $A,B \in \cF$.  If $A,B \in \cE_K$ for some $K$, it follows that $A \subseteq B$, $B \subseteq A$, or $|A \cap B| < t$.  On the other hand, if $A$ and $B$ each belong to different $\cE_{K_1}$ and $\cE_{K_2}$, then $|K_1 \cap K_2| < t$ since $\cB$ is a $t$-packing.  It follows that $|A \cap B| < t$ in this case.
\end{proof}

In the notation of our main result, Theorem~\ref{main}, recall that we are letting $f(n)$ denote the maximum $|\cF \cap \binom{[n]}{\ge 2}|$ over all 2-laminar families $\cF$. 

\begin{cor}
\label{2lb}
Suppose there exists a $2$-$(n,m,1)$ packing with $b$ blocks.  Then
$f(n) \ge b \cdot f(m) + 1$.
\end{cor}

Note the `$+1$' is due to  the universal set $[n]$.
Now, Construction~\ref{nest} can be used inductively.  The choice of packings (or designs) and block sizes must now be explored further.  To this end, we recall two famous families of designs.

\begin{lemma}
\label{ap-pp}
There exists a $2$-$(q^2,q,1)$ design and a $2$-$(q^2+q+1,q+1,1)$ design for every prime power $q$.
\end{lemma}

The designs stipulated in Lemma~\ref{ap-pp} are, respectively, the affine and projective planes of order $q$.  In particular, the projective plane for $q=2$ is the well known \emph{Fano plane}, a 2-$(7,3,1)$-design.  They lead to an easy construction of $2$-laminar families for an infinite sequence of $n$.

\begin{prop}
\label{seven}
Let $n = 7^{2^r}$.  Then 
\begin{equation}
\label{7bd}
f(n) \ge \binom{n}{2} \left[1+\frac{1}{\binom{3}{2}}+\frac{1}{\binom{7}{2}}+ \frac{1}{\binom{49}{2}} + \dots + \frac{1}{\binom{n}{2}} \right].
\end{equation}
\end{prop}

\begin{proof}
Use induction on $r$.  If $r=0$, then $n=7$ and we construct a family $\cF_0:= \binom{[7]}{2} \cup \cB \cup \{[7]\}$,
where $\cB \subseteq \binom{[7]}{3}$ are the $\binom{7}{2}/\binom{3}{2}$ blocks of a Fano plane.

It is easy to check that $\cF_0$ is $2$-laminar, leading to $f(7) \ge 21+7+1 = 29$.
Now for $i \ge 1$, construct $\cF_i$ from $\cF_{i-1}$ by applying Construction~\ref{nest} to the $2$-$(7^{2^i},7^{2^{i-1}},1)$ design in Lemma~\ref{ap-pp}.  It follows that $\cF_r$ is $2$-laminar with ground set $[n]$.  It remains to count $|\cF_r|$.

In general, there are $\binom{n}{2} \binom{m}{2}^{-1}$ blocks in a $2$-$(n,m,1)$ design.  With $m=7^{2^{r-1}}$ and $n=m^2$,
we have $$\frac{f(n)}{\binom{n}{2}} \ge \frac{|\cF_r|}{\binom{n}{2}} =  \frac{|\cF_{r-1}|}{\binom{m}{2}} +\frac{1}{\binom{n}{2}}.$$
We obtain the (\ref{7bd}) after an easy induction.
\end{proof}

To obtain a construction for general $n \gg 0$, we make use of an asymptotic result on maximum 2-packings.

\begin{lemma}[\cite{CCLW}]
\label{asym-pack}
Given any $k$, there exists a universal constant $c_k$ such that, for all $n$, the maximum size of a $2$-$(n,k,1)$ packing is at least $\lfloor \frac{n}{k} \lfloor \frac{n-1}{k-1} \rfloor \rfloor - c_k$.
\end{lemma}

We are now ready to prove the first half of our main result.
\begin{proof}[Proof of Theorem~\ref{main}{\rm (a)}]
Let $k=7^4$.  It follows from Proposition~\ref{seven} and a quick calculation that $f(k) \ge 1.3818 \binom{k}{2}$.  Now, take an $(n,k,1)$-packing, which by Lemma~\ref{asym-pack} has at least $b=\binom{n}{2} \binom{k}{2}^{-1}-o(n^2)$ blocks.  It follows from Corollary~\ref{2lb} that
$f(n) >  [1.3818-o(1)] \binom{n}{2}$, as desired.
\end{proof}
 
For $t=3$, there is a classical family of designs we can use to take the place of the finite planes.
\begin{lemma}
\label{circle}
There exists a $3$-$(q^2+1,q+1,1)$ design for every prime power $q$.
\end{lemma}
These designs are known as \emph{circle geometries}.  They arise from the action of PGL$(2,q^2)$ on $\F_{q} \cup \{\infty\}$ with its natural containment in $\F_{q^2} \cup \{\infty\}$.  Using the same construction as for $t=2$, we have the following.

\begin{prop}
\label{three}
Let $n=3^{2^r}+1$.  There exists a $3$-laminar family of $[n]$ with size
\begin{equation}
\label{3bd}
1+\binom{n}{1} + \binom{n}{2} + \binom{n}{3} \left[1+\frac{1}{\binom{4}{3}}+\frac{1}{\binom{10}{3}}+ \dots + \frac{1}{\binom{n}{2}} \right] \approx 1.5083 \binom{n}{3}.
\end{equation}
\end{prop}

Incidentally, the new result of Keevash \cite{Keevash} provides an analog of Lemma~\ref{asym-pack} for general $t$.  Thus we can expect (\ref{3bd}) as an asymptotic lower bound on $3$-laminar families for all integers $n \gg 0$, not just $n=3^{2^r}+1$.  For $t>3$, the difficult part is finding good $t$-$(n,k,1)$ packings where $n$ is not too much larger than $k$.  This is the fundamental challenge to going further.  All we can say in general is that $t$-laminar families exist with size at least $(1+\epsilon) \binom{n}{t}$, where $\epsilon>0$ depends on $t$.

\section{The upper bound: linear programming}

We begin with a structural lemma on $2$-laminar families.

\begin{lemma}
\label{struct}
Let $\cF \subseteq \binom{[n]}{\ge 2}$ be a $2$-laminar family and define $\cF^*$ to be the set of maximal elements of $\cF \setminus \{[n]\}$.  Let there be $b_i$ elements of $\cF^*$ of cardinalty $i$.  Then
$$|\cF| \le 1+\sum_{2 \le k < n} f(k) b_k.$$
\end{lemma}

\begin{proof}
Every element of $\cF \setminus \{[n]\}$ is contained in a unique element of $\cF^*$.  For each $K \in \cF^*$, the restriction of $\cF$ to sets contained in $K$ is a $2$-laminar family, say $\cF|_K$. But $|\cF|_K| \le f(|K|)$, and so the result follows by summing on $K$.
\end{proof}

Next, recall from Proposition~\ref{packings} that the maximal elements of $\cF \setminus \{[n]\}$ form a $2$-packing on $[n]$.  In this way, we can constrain the parameters $b_k$ appearing above.

\begin{lemma}
Let $([n],\cB)$ be a $2$-packing with $b_k$ blocks of size $k$, $2 \le k \le n$.  Then 
\begin{equation}
\label{all-pairs}
\sum_{2 \le k \le n} \binom{k}{2} b_k \le \binom{n}{2}.
\end{equation}
Moreover, if $b_m>0$ and $b_k=0$ for $m < k \le n$, we have
\begin{equation}
\label{non-pairs}
\sum_{2 \le k \le m} \binom{k-1}{2} b_k \le \binom{n-m}{2}+\binom{m-1}{2}.
\end{equation}
\end{lemma}

\begin{proof}
The inequality (\ref{all-pairs}) follows simply by counting $\binom{[n]}{2}$ in two ways.  
Next, let $B \in \cB$ with $|B|=m$ and count $\binom{[n] \setminus B}{2}$.  By the packing condition, every other block of size $k$ intersects $[n] \setminus B$ in at least $k-1$ points.  It follows that
$$\sum_{K \in \cB, K \neq B} \binom{|K|-1}{2} \le \binom{n-m}{2},$$ 
from which (\ref{non-pairs}) easily follows.
\end{proof}

Using just Lemma~\ref{struct} and inequality (\ref{all-pairs}), we arrive at a recursive upper bound.
\begin{prop}
\label{rec-bound}
For all $n > 2$, $$\frac{f(n)}{\binom{n}{2}} \le \frac{1}{\binom{n}{2}} + \max_{2 \le k < n} \frac{f(k)}{\binom{k}{2}}.$$
\end{prop}

\begin{proof}
$$\frac{f(n)}{\binom{n}{2}} \le \frac{1}{\binom{n}{2}} +\sum_{2 \le k < n} \frac{f(k)}{\binom{n}{2}} b_k = \frac{1}{\binom{n}{2}} + \frac{1}{\binom{n}{2}} \sum_{2 \le k < n} \frac{f(k)}{\binom{k}{2}} \cdot \binom{k}{2} b_k \le \frac{1}{\binom{n}{2}} + \max_{k < n} \frac{f(k)}{\binom{k}{2}}.\qedhere$$
\end{proof}

\begin{cor}
\label{easy-bound}
For all $n > 2$, $f(n) \le 2\binom{n}{2}$.
\end{cor}

\begin{proof}
Note that $f(3)=4$.  For $n>3$, Proposition~\ref{rec-bound} gives
$$\frac{f(n)}{\binom{n}{2}} \le \frac{f(3)}{\binom{3}{2}} + \sum_{k >3} \binom{k}{2}^{-1} = 2.\qedhere$$
\end{proof}

To use (\ref{non-pairs}), we consider $m$ as the maximum cardinality of a set in $\cF^*$.  On the one hand, larger values of $m$ contribute a larger $f(m)$ to the bound on $f(n)$.  On the other hand, larger values of $m$ constrain the packing via (\ref{non-pairs}).  To explore this trade-off, we set up a (recursive) linear program.  

Define a function $\obf: \Z_{\ge 2} \rightarrow \Q$ by $\obf(2)=1$, $\obf(3) = 4$ and, for $n>3$,
$$\obf(n)=1+\max_{2 \le m < n}\LP(n,m),$$
where $\LP(n,m)$ is the solution to the linear program
\begin{equation}
\label{lp}
\begin{array}{lll}
\mbox{maximize:} & \sum_{2 \le k \le m} \obf(k) b_k \\
\mbox{subject to:} & b_m \ge 1, & b_k \ge 0 \text{ for } 3 \le k<m, \\
&(\ref{all-pairs}) \text{ and } (\ref{non-pairs}).\\
\end{array}
\end{equation}

Since this is a fractional relaxation of Lemma~\ref{struct} and the packing constraints hold, we have an upper bound on $f(n)$.
\begin{prop}
For all $n > 2$, $f(n) \le \obf(n)$.
\end{prop}

The linear program (\ref{lp}) is somewhat tricky to analyze because of the recursive dependence on $\obf$.  It is probably easiest to just appeal to a long computation which stores $\obf$ in memory.

\begin{proof}[Proof of Theorem~\ref{main}{\rm (b)}]
We compute $\obf(50000) \binom{50000}{2}^{-1} \approx 1.38206$.  Then, appealing to Proposition~\ref{struct} we have
$$f(n) \binom{n}{2}^{-1} \le 1.38206 + \sum_{k> 50000} \binom{k}{2}^{-1} \approx 1.3821. \qedhere$$
\end{proof}

It is worth providing some additional details on the computation.  First, let us dualize and slightly rewrite (\ref{lp}) so that LP$(n,m)$ is also the solution of the (now two-variable) program
\begin{equation}
\label{lp-dual}
\begin{array}{lll}
\mbox{minimize:} & \binom{n-m}{2}x + (\binom{n}{2}-\binom{m}{2}) y \\
\mbox{subject to:} & x \ge 0, y \ge 1, \text{ and}  \\
& \binom{k-1}{2} x + \binom{k}{2} y \ge \obf(k) \text{ for } k = 3,\dots,m.\\
\end{array}
\end{equation}
A major advantage now is that (\ref{lp-dual}) seems to have many redundant constraints.  For $k \ge 2$, let $\eta_k$ denote the halfspace 
$$\eta_k = \left\{(x,y) : \binom{k-1}{2}x + \binom{k}{2} y \ge \obf(k) \right\}.$$
For example, $\eta_2$ is defined by $y \ge 1$.  As a special definition, let $\eta_1$ be the right halfspace defined by $x \ge 0$.
As we build a table of $\obf(n)$, we keep track of the intersection 
$$\Theta_n := \bigcap_{k=1}^n \eta_k,$$ which is the feasible region for (\ref{lp-dual}).  When $n$ is incremented, it is a simple check whether $\Theta_n \subseteq \eta_{n+1}$.  If so (and this almost always happens), the constraint $\eta_{n+1}$ is eliminated.  For instance, we compute 
$$\Theta_{10000} = \eta_1 \cap \eta_2 \cap \eta_3 \cap \eta_7 \cap \eta_{43} \cap \eta_{1807},$$
and we note that these subscripts (starting with 2) are obtained by iterating the map $k \mapsto k^2-k+1$.  These are the sizes of hypothetical `nested' projective planes. (Owing to the nonexistence of the $2$-$(43,7,1)$ design, though, our construction in Section 2 must shift to use powers of 7.  This is essentially what accounts for the small gap between our bounds.)  In any case, eliminating constraints at each stage reduces the calculation of LP$(n,m)$ to a minimum of evaluations of the objective function at the handful of extreme points of $\Theta_m$.  This is efficient enough for our purposes to sidestep a more rigorous analysis.
However, assuming one can prove the pattern of critical halfspaces continues, a slightly sharper upper bound than Theorem~\ref{main}(b) is possible, namely
$f(n) \binom{n}{2}^{-1} \le 1+\binom{3}{2}^{-1} + \binom{7}{2}^{-1} +\binom{43}{2}^{-1} + \dots$.

The technique in this section can, in principle, extend to $t>2$, though the bounds likely far exceed what we can construct.  A general upper bound of $O(n^t)$ on $t$-laminar families follows by a similar argument as in Corollary~\ref{easy-bound}, and that is enough for an order-of-magnitude determination.

\section*{Acknowledgements}
Thanks to Richard P.~Anstee for helpful discussions and suggestions.

\end{document}